\documentclass[american]{IEEEtran}
\usepackage[T1]{fontenc}
\usepackage{color}
\usepackage{float}
\usepackage{calc}
\usepackage{amsmath}
\usepackage{amssymb}
\usepackage{graphicx}          
\usepackage{float}
\usepackage{pifont}
\usepackage{bbding}  
\usepackage{times}
\usepackage{esvect}
\usepackage{nomencl}
\usepackage{amsmath,amssymb,epsf,epsfig,times}
\usepackage{amsfonts}
\usepackage[all]{xy}
\usepackage{color}
\usepackage{subfigure}
\usepackage{url,cite}
\usepackage{algorithm}
\usepackage{algpseudocode}
\usepackage{bm}

\IEEEoverridecommandlockouts

\newtheorem{theorem}{Theorem}
\newtheorem{lemma}{Lemma}

\newtheorem{corollary}{Corollary}
\newtheorem{assumption}{Assumption}
\newtheorem{definition}{Definition}
\newtheorem{remark}{Remark}

\def\QED{\mbox{\rule[0pt]{1.5ex}{1.5ex}}}

\newcommand{\norm}[1]{\left\|#1\right\|}

\newcommand{\abs}[1]{\left|#1\right|}

\newcommand{\OMIT}[1]{}

\newif\ifpdf
\ifx\pdfoutput\undefined \pdffalse \else \pdfoutput=1 \pdftrue \fi
\ifpdf
\usepackage{graphicx}
\usepackage{epstopdf}
\usepackage{graphicx}
\fi
\begin{document}
\title{Distributed Adaptive Reinforcement Learning:\newline A Method for Optimal Routing}

\author{Salar Rahili, Benjamin Riviere, and Soon-Jo Chung\thanks{California Institute of Technology, Pasadena, CA, 91125}}

\maketitle


\begin{abstract}
	In this paper, a learning-based optimal transportation algorithm for autonomous taxis and ridesharing vehicles is presented. 
	The goal is to design a mechanism to solve the routing problem for multiple autonomous vehicles and multiple customers in order to maximize the transportation company's profit.
	As a result, each vehicle selects the customer whose request maximizes the company's profit in the long run.
	To solve this problem, the system is modeled as a Markov Decision Process (MDP) using past customers data.
	By solving the defined MDP, a centralized high-level planning recommendation is obtained, where this offline solution is used as an initial value for the real-time learning.
	Then, a distributed SARSA reinforcement learning algorithm is proposed to capture the model errors and the environment changes, such as variations in customer distributions in each area, traffic, and fares, thereby providing optimal routing policies in real-time.
	Vehicles, or agents, use only their local information and interaction, such as current passenger requests and estimates of neighbors' tasks and their optimal actions, to obtain the optimal policies in a distributed fashion.
	An optimal adaptive rate is introduced to make the distributed SARSA algorithm capable of adapting to changes in the environment and tracking the time-varying optimal policies.
	Furthermore, a game-theory-based task assignment algorithm is proposed, where each agent uses the optimal policies and their values from distributed SARSA to select its customer from the set of local available requests in a distributed manner.
	Finally, the customers data provided by the city of Chicago is used to validate the proposed algorithms.
 \end{abstract}

\section{ Introduction}	

Urban transportation plays a significant role in the development of modern cities. Almost 1.2 million deaths occur on roads each year worldwide, and reports show that $94\%$ of car accidents in the U.S. involve human errors \cite{Google}. Autonomous cars are an emergent technology that will quickly become ubiquitous as a safer and more efficient mode of transportation. Transportation Network Companies (TNCs) are planning to employ coordinated fleets of autonomous ground and air vehicles to improve the urban transportation capabilities \cite{Uber}, see Fig.~\ref{overview_graphic}. The deployment of fleets of autonomous vehicles, both ground and air, drives a coupled innovation in algorithm development. 


\begin{figure}
\begin{center}
\includegraphics[width=75mm]{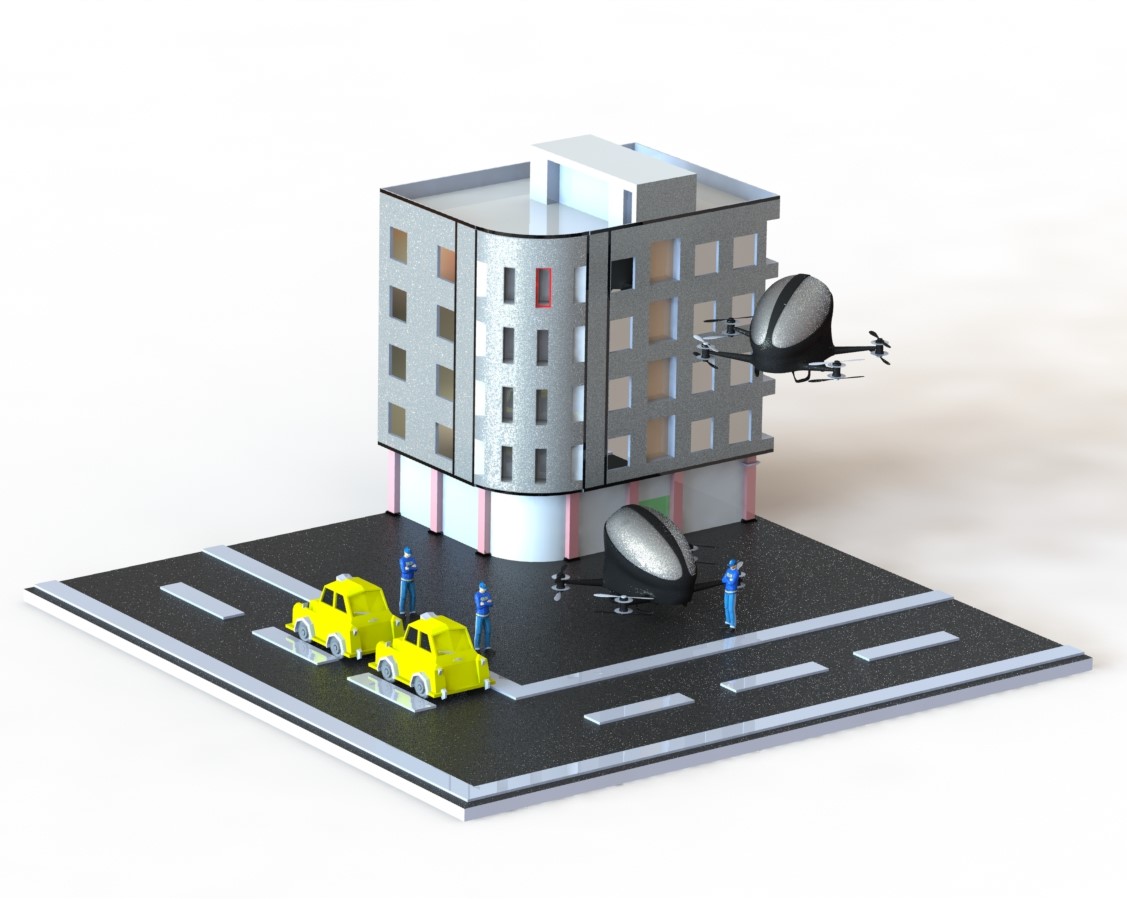}
\caption{
Concept graphic of an intelligent transportation network. Agents, both ground and air vehicles, estimate in real-time the state of the environment and select optimal customers to maximize the transportation company's profit.
}\label{overview_graphic}
\end{center}
\end{figure}

Planning such an on-demand transportation system that must adapt to customer needs in real-time has been studied in the literature. The problem of providing transportation services for customers can be modeled as a Pick-up and Delivery Problem (PDP) \cite{StaticSur} or its extension Dial-A-Ride Problem (DARP) \cite{DARP} in which the transportation of goods is replaced by the transportation of people. Most prior work in the literature is focused on a static routing problem, where all the customers' requests for all time are known before routes are determined. However, due to the fast dynamics of customers requests and unknown future requests, employing these methods for planning real-time transportation in urban areas is not possible. Recently, some new research has been conducted on dynamic and stochastic routing using PDPs, where part or all of their input is unknown and revealed dynamically \cite{DynHv,DynFra,DynamicSur,dynamic_congestion}. The main objective of these studies is to minimize the total distance traveled by the vehicle while servicing all customers. 
In \cite{optimal_routing}, a Markov Decision Process formulation optimizes cost and vehicle usage in a dynamic environment. However, the solutions in the literature are mostly addressed by proposing centralized methods. In \cite{multi}, a centralized algorithm for routing problem is introduced, where the computation cost limits its ability to solve the problem for only $5$ vehicles and $17-25$ customers.

Recent works propose scalable solutions to dynamic routing problems.
In \cite{anticipatory_routing}, a decentralized solution is presented to minimize traffic congestion using an ant-pheromone-inspired method.
In \cite{deep_rl}, a distributed deep reinforcement-learning method is proposed to learn macro-actions in event-response, useful for dynamic routing. A multi-agent MPC-based control is introduced in \cite{cyber_car}, where a fleet of autonomous taxis is controlled in a scalable manner to minimize total energy cost. In contrast to these frameworks, we proposed a distributed reinforcement learning algorithm with learning-rate adaption and a dynamic consensus algorithm to control fleets of taxis to maximize company profit by optimally selecting customers while only a limited amount of information must be communicated among local neighbors. 
 A limitation of prior routing studies is that the capacity of the vehicle is limited to only one customer; hence, ride-pooling capability is not considered. In ride-pooling, two or more customers can be matched to get service simultaneously by one vehicle. This can significantly reduce the cost for customers and also reduce the number of required vehicles for the transportation company. Existing work in ride-sharing from \cite{intelligent_carpool} mines GPS data to find a set of frequent routes to intelligently propose ride-sharing routes in real-time. We propose a game-theoretic ride-sharing extension in our task assignment.

\begin{figure*}[t]
\begin{center} \hspace*{0cm}
\vspace{0.1cm} {\scalebox{0.5}{\includegraphics*{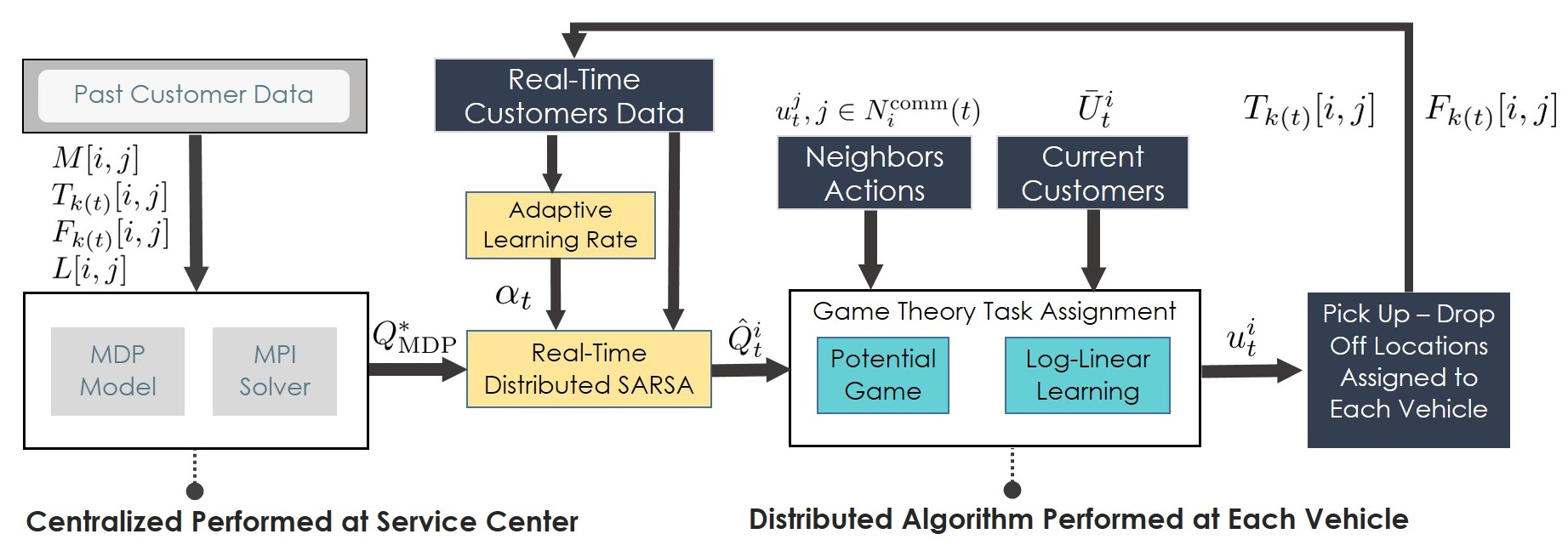}}}
\vspace{0cm} \caption{Schematic of the proposed optimal routing algorithm.}
\vspace{-0.4cm}
\label{Schematic}\end{center}
\end{figure*}

In this paper, we present a distributed learning-based optimal traffic planning and decision making algorithm that integrates planning with a local decision making algorithm. The proposed scheme performs in a distributed fashion based on local information. Such local information includes other neighboring vehicles' route tasks, their estimate of the optimal actions, and current passenger requests. Here, our goal is to design some mechanisms to find the optimal actions for these two kinds of autonomous transport vehicles. We propose a unified strategy to solve the routing problem to maximize the transportation company's profit. To attain this goal, each vehicle selects the best customer among current requests in order to maximize the company's profit in the long run. The proposed algorithm is shown in Fig.~\ref{Schematic}. The research directions proposed in review study \cite{its_review} is well aligned with our algorithm to coordinate mobile agents in a distributed fashion to handle uncertainty in a dynamic environment.

In Sec. \ref{SubSec:MDP-taxi}, the problem is modeled as a Markov Decision Process (MDP). The past customers data is used to predict the probability of having customers and their trips in each area. Solving the MDP provides the agents a high-level planning recommendation, including a list of state transitions and their potential values. This is shown in gray boxes of Fig.~\ref{Schematic}. However, the static solution built only based on past data is not accurate, as it is not able to capture any changes in the environment. Reinforcement learning can be used as a decision making scheme when an accurate model of the environment is not known. 

In Sec.~\ref{subsec:SARSACentralized}, a State-Action-Reward-State-Action (SARSA) reinforcement learning algorithm is presented, which allows the system to learn its model (i.e., transition probabilities and rewards) and update the optimal policies while the optimal policies obtained from the MDP solution are used as an initial value. The conventional reinforcement learning usually is not well-suited for a non-stationary environment, and the commonly-used proofs of convergence hold only for stationary environments. In our problem setting, the environment is non-stationary. In particular, the number of the customers in each area, traffic and fares are changing over time. 

In Sec.~\ref{SubSec:adaptive}, we propose an optimal adaptive learning-rate tuning, and modify the SARSA reinforcement learning in order to track the environment changes in non-stationary environments. Moreover, the SARSA algorithm is a centralized algorithm, where all information is required to be sent to a central node to be fused. However, in our framework with many vehicles and trips in each step, it is not feasible to pass all information to and from a command center. 

Hence, a fully distributed SARSA reinforcement learning is proposed in Sec. \ref{SubSec:Dis-Learning}, where agents are only using their own local information and local interactions to update the optimal policies of the system. The proposed modified distributed SARSA provides the value of each action in the environment. Each vehicle uses these values to evaluate each customer. In a single agent scenario, the agent simply selects the customer with the largest value. However, in a multi-vehicle scenario, agents are required to reach an agreement on selected customers in order to avoid any conflict. To solve this problem, agents need to agree on how to distribute the customers among themselves. In Sec.~\ref{Sec:game}, a real-time task assignment algorithm based on game theory is proposed to enable the agents to select their non-conflicting tasks/customers in a distributed fashion. In Sec.~\ref{SubSec:ridesharing}, we focus on ride sharing and courier taxi service routing. As compared to our preliminary work presented in an eight-page-long workshop article~\cite{priorWork}, this paper include many revision in all the sections, including two additional mathematically-rigorous proofs of convergence, more complete proofs of the theorems, and an appendix detailing some math used in the main proofs.

\section{Preliminaries and Definitions} \label{Sec:Prelim}
\subsection{Notation}
The following notations are adopted throughout this paper. 
	$x \in \mathbb{R}^n$ denotes the vector of size $n$. Let $\textbf{1}_n$ and $\textbf{0}_n$ denote the column vectors of $n$ ones and zeros, respectively. 
	The Cartesian product of sets, $S_i, \forall i$ is denoted by $\prod_{i=1}^N S_i$.
	The empty set is denoted by $\varnothing$.
	The cardinality of a set $S$ is denoted by $|S|$.
 	$\mathbb{E}[\cdot]$ is the expectation operator.
    $\mathbb{V}\text{ar}$ is the variance of a random variable. 
	The arithmetic mean of a series of numbers, $a_i$, for $i=1,\cdots,n$, is denoted by $\text{Avg}(a_i)=\frac{1}{n}\sum_{i=1}^{i=n} a_i$.
	The exponential function is written as $\exp(\cdot)$.
	A time-varying digraph $\mathcal{G}_t \triangleq (V,E_t)$ is used to characterize the interaction topology among the agents, where ${V \triangleq \{1,\ldots,N\}}$ is the node set and $E_t \subseteq V \times V$ is the edge set. An edge $(i,j) \in E_t$ means that node $i$ can obtain information from node $j$ at time $t$.  The adjacency matrix ${\mathcal{A}(t) \triangleq [\mathcal{A}_{ik}(t)] \in \mathbb{R}^{N \times N}}$ of the graph $\mathcal{G}_t$ is defined such that the edge weight ${\mathcal{A}_{ik}\neq 0}$ if ${(j,i) \in E_t}$ and ${\mathcal{A}_{ik}=0}$ otherwise.
	The compact two-dimensional space (i.e., city map) is partitioned into $n_q \in \mathbb{N}$ disjoint partitions. The size of the cells is selected by the designer based on the desired spatial resolution and the computation expenses. The superscript is an agent index, and the subscript is a time index. 

\subsection{Markov Decision Process Formulation} \label{SubSec:MDP-taxi}
MDP is a mathematical framework introduced to make decisions in a stochastic, known environment (see \cite{Puterman,bellman1957} and references therein), and the solution of MDP is a policy providing all optimal actions in each state of the environment.
In our problem setting, the current customers at each cell are time-varying and not known a priori. Hence, MDP cannot be used as an actor to adapt itself quickly and obtain an optimal policy for current possible actions. However, based on the stochastic model embedded in MDP, we are able to evaluate the profit gained by selecting each  possible action in the long run. Thus, we use MDP to estimate the value of each possible action on each state, instead of an actor to select the optimal policies directly. Then, an agent can easily use the provided estimation from MDP and select the optimal current customer.
		
In our problem framework, at each time instant each vehicle knows the requests of the current local customers and, by getting help from our decision making algorithm, can deterministically select a customer among available customers. However, the future customers requests at the destination are unknown. We first propose an MDP to model the problem. Then solving the defined MDP problem provides a high-level policy making recommendation for each agent. These recommendations include a list of ranked possible actions for each cell (i.e., vehicle's current cell), and a value corresponding to the expected infinite horizon average payoff for each action. Here, we iteratively modify our model and solve the MDP for each state. In particular, in an inner-loop $1\leq l \leq n_q$, it is assumed that the agent is in cell $l$, and aware of its local current customers requests, while  the set of possible actions in future time instants are not known.
\begin{definition}\label{reward}
An agent earns $F_{k(t)}[i,j]$ fare for task $k$ at time $t$, where the task is completed by picking up and dropping off the customer from cell $i$ to cell $j$, respectively. The reward is denoted by $D_{k(t)}[i,j]$, and can be calculated as $D_{k(t)}[i,j]=\frac{M[i,j]}{T_{k(t)}[i,j]}F_{k(t)}[i,j]$, where $M[l,j]$ corresponds to any motion constraints to go from cell $l$ to cell $j$, and $T[l,j]$ corresponds to time to go from cell $l$ to cell $j$. It is assumed that agents are not moving if they are not assigned to a  customer; hence the reward for an agent is assumed to be zero if it does not have a customer. An agent might remain in its current cell in two cases: 1) not having any customer at time $t$, which means $k(t)=$ \O; hence, $D_{k(t)}[i,i]=0$, or 2) having a customer such that the pick up and delivery points are both inside cell $i$, then the reward is $D_{k(t)}(i,i)\neq 0$. To accommodate this effect, we denote the average of rewards for all similar tasks as $D[i,j]=\text{Avg}(D_{k(t)}[i,j]), \forall i,j$. Note that $D[i,i]$ contains the average of both zero and nonzero rewards. 
\end{definition}
\hspace{0.25cm} Our MDP is formulated with a tuple, $<S,A,\mathcal{P},\mathcal{R}>$ as follows:
\begin{itemize}
\item State variables $S$:  The finite set of zones or cells in the city, denoted by $S=\{i \mid \forall i \in 1, 2, \cdots, n_q\}$. 
\item Actions $A$: The set of possible actions at cell state $l$ is  $A(l)=\{ a^l_j\}$, where $a^l_j$ is the action of moving into cell $j$ from cell $l$. 
\item Reward  model $\mathcal{R}$: 
\begin{align*}
\mathcal{R}_{a^i_j}(i,j) = 
\begin{cases} 
   D[i,j] & \forall i, i \neq j \\
   \frac{L[i,i] D[i,j]}{1 + L[i,i] - \sum_k L[i,k]} & i = j\\ 
\end{cases}
\end{align*} where $D[i,j]$ is the average reward defined in Definition \ref{reward}, and $L[i,j]$ is the probability of having a customer to pick up from cell $i$ and deliver to cell $j$. 
\item Transition probabilities $\mathcal{P}$:  
\begin{align*}
\mathcal{P}_{a^i_j}(i,j)=
\begin{cases} 
   L[i,j] & i\neq j \\
   1 + L[i,i] - \sum_j L[i,j] & i = j\\ 
\end{cases}
\end{align*} where $L[i,j]$ is the probability of having a customer to pick up from cell $i$ and deliver to cell $j$.
\end{itemize}

 {\color{black} To ensure that the defined model is a proper MDP, we show that $\sum_{k=1}^{k=n_q} \mathcal{P}_{a^l_{j}}(l,k)=\sum_{k=1}^{k=n_q} \mathcal{P}_{\bar{a}^i}(i,k)=1$, for  any inner-loop $l$, and any action.
 \begin{itemize}
 \item[I)]Assume that the current state is $l$ and we have full action set of $A(l)$. The algorithm deterministically chooses an action $\mathcal{P}_{a^l_{j}}$; hence, $\sum_{k=1}^{k=n_q} \mathcal{P}_{a^l_{j}}(l,k)= \mathcal{P}_{a^l_{j}}(l,j)+ \sum_{k \neq j} \mathcal{P}_{a^l_{j}}(l,k) =1+0=1.$

 \item[II)]  Now, assume that we are in the $l$th loop but currently in a state $i\neq l$, where we have only one action $\bar{a}^i$, then $\sum_{k=1}^{k=n_q} \mathcal{P}_{\bar{a}^i}(i,k)= \sum_{k=1,k \neq i}^{k=n_q} \mathcal{P}_{\bar{a}^i}(i,k)+\mathcal{P}_{\bar{a}^i}(i,i) = \sum_{k=1,k \neq i}^{k=n_q} L[i,k]+1+L[i,i]- \sum_{k=1}^{k=n_q} L[i,k]=1.$
 \end{itemize}

We formulate the maximum reward problem with a Q-value for a state-action pair, $Q(i,\pi[i])$, a value function $V(i)$, and a policy $\pi[i]$, defined with terms from the MDP tuple. This is a dynamic programming problem, where $\mathcal{R}(i,j)$ is the immediate reward to go from cell $i$ to cell $j$, $V(j)$ is all the future reward if actions are chosen optimally from cell $j$, and $\gamma$ is the discount factor that penalizes future rewards exponentially. 
\begin{align}
Q(i,{\pi[i]})&= \sum_{j\in n_q} \mathcal{P}_{\pi[i]}(i,j) \big(\mathcal{R}_{\pi[i]}(i,j)+\gamma V(j)\big),\label{Q-value}\\
V(j)&= \max_{\pi[j] \in A(j)}Q(j,{\pi[j]}) , \label{V}\\
\pi^*[i]&=\arg \max_{{a^i_j} \in A(i)} Q(i,{{a^i_j}}), \label{maximization}
\end{align}
A solution to \eqref{maximization} is an optimal policy, denoted by $\pi^*[i]$, defined by the Bellman equation. Note that this optimal policy can be obtained knowing the Q-value, $Q(.)$, for each state-action pair. To find the solution for \eqref{maximization}, we use a Modified Policy Iteration (MPI) algorithm to estimate \eqref{V} through several steps of successive approximation. The optimal solution of the MDP problem is aggregated as a vector denoted by $Q^*_{\text{MDP}}$.

Here, our goal is to solve the MDP problem while we keep the value of Q-function $Q(l,\cdot)$ in our memory for each available action at cell $l$.
	The optimal solution of the MDP problem is aggregated as a vector denoted by $Q^*_{\text{MDP}}\in\mathbb{R}^{\sum_{k=1}^{n_q} |A(k)|}$.
	Note that in our framework the optimal policy $\pi[i]$ calculated in \eqref{maximization} is not necessarily available for agent at time $t$.
	Each vehicle can only select a customer among the set of pick-up requests available at time $t$.
	Hence, we not only keep the stationary optimal policy $\pi[l]$ in our memory, but we will also save a list of all available actions at each cell $l$ ranked based on the value of Q-function.

\section{Distributed SARSA Reinforcement Learning with Adaptive Learning Rate Tuning} \label{Sec:SARSA}
    To account for the changing environment, we begin with the near-optimal $Q^*_\text{MDP}$ and update the optimal policy on-line at each time step using new data from customers and agents' trips. The error of the MDP solution from the dynamic probability and reward models is shown in Appendix \ref{Appendix-errorbound}, which demonstrates the need to complement the MDP solution with learning. State-Action-Reward Reinforcement Learning (SARSA RL) is used because it can obtain an optimal policy when the system's model (i.e., $\mathcal{P}$ and $\mathcal{R}$) is not known in advance. First, we present a centralized SARSA RL algorithm in Sec. \ref{subsec:SARSACentralized}. Then, we present an optimal adaptive learning rate in Sec. \ref{SubSec:adaptive}. Finally, we present a novel distributed SARSA RL algorithm with a proof of convergence in Sec.~\ref{SubSec:adaptive}. 
    
\subsection{Centralized SARSA RL for Stationary Environment} \label{subsec:SARSACentralized}
First, we present a standard, centralized model-free reinforcement learning as a contextual comparison for our main contribution of deriving distributed algorithms. The Q-values with respect to state-action pairs are updated in a SARSA RL framework as:
\begin{align}
Q_{t+1}(i,\pi[i])&= (1-\alpha_t)Q_t(i,\pi[i])\nonumber \\&+\alpha_t \big(\mathcal{R}_{\pi[i]}(i,j)+\gamma Q_t(j,\pi[j])\big),
 \label{Q-update}
\end{align}
where $\alpha_t$ is a learning rate satisfying Remark \ref{alpha}, and $Q(\cdot)$ is updated under policy $\pi[i]$ to transition from cell $i$ to cell $j$. This formulation means that the action $\pi[j]$ at the successor state $j$ is not necessarily optimal, while in Q-learning, the successor action is chosen to be optimal.
\begin{remark}
Assume we are accomplishing the $k^\text{th}$ task at time $t$ by selecting action $\pi[i]$, where this task is completed by picking up a customer from cell $i$, and dropping him/her off at cell $j$. Then, the reward function is defined as $\mathcal{R}_{\pi[i]}(i,j)=D_{k(t)}[i,j]$. It can be seen that the reward function for SARSA in each time $t$ will be coherent with its average defined in Definition \ref{reward}, and used in MDP model in Sec. \ref{SubSec:MDP-taxi}.
\end{remark}
\begin{remark}\label{alpha}
In conventional reinforcement learning, the sequence of otherwise arbitrary $\alpha_t$ satisfies: $\sum_{T=0}^{\infty} \alpha_t= \infty$ and $\sum_{T=0}^{\infty} \alpha_t^2< \infty$. The Q-values eventually converge to a constant as the update term goes to zero. 
\end{remark} 
\hspace{0.25cm} However, in a non-stationary environment, we want the adaptive learning rate to not converge to zero, such that $Q_t(.)$ value can continue being updated in \eqref{Q-update}. In the next section, we will introduce a method to estimate the optimal adaptive learning rate dynamic signal.
    
\subsection{Adaptive Learning-Rate for Non-stationary Environment} \label{SubSec:adaptive}
In this subsection, a new algorithm (shown in yellow color in Fig.~\ref{Schematic}) is presented in Theorem \ref{Theorem-adaptive-alpha}, to estimate the optimal learning rate, $\alpha_t(i,{\pi[i]})$, at each time $t$ and for each action-state pair. Estimating the optimal policy is equivalent to converging the estimated Q-value to the Q-value at the next time-step. For each new sample data from a new customer with state-action pair $(i,{\pi[i]})$ at time $t$, the Q-update, denoted $v_t(i,\pi[i])$, and its stochastic information can be written as:
\begin{align}
    v_t(i,{\pi[i]}) &= \mathcal{R}_{\pi[i]}(i,j)+\gamma Q_t(j,{\pi[j]}), \label{update_randomness_eqn}\\ 
    \mathbb{E}(v_t) &= Q_t^*(i,{\pi[i]}), \ \ \mathbb{V}\text{ar}(v_t) = \sigma_t^2(i,{\pi[i]}), \nonumber
\end{align}    
Note that $Q_t^*(i,{\pi[i]})$ is the equilibrium solution of \eqref{Q-update}, which can be obtained by computing the expectation of the sample, $v_t$.

Here, we make the following assumption on variance of update for a state-action pair, $\sigma_{t}$.
\begin{assumption} \label{sigma-timeinvariant}
The observations' covariance of each state-action pair is assumed to be time-invariant, i.e., ${\sigma_{t+1}}(i,{\pi[i]})={\sigma_t}(i,{\pi[i]}), \forall t$, and it will be written as ${\sigma}(i,{\pi[i]})$.
\end{assumption} 

The mean and variance of the observation, $v_t$, are simply functions of probabilistic distribution of $\mathcal{R}_{\pi[i]}(i,j)$. In our problem setting, the reward model, $\mathcal{R}_{\pi[i]}(i,j)$ is determined by the ratio of fare, ${F[i,j]}$, and time to complete request, $T[i,j]$. Note that the mean of the reward function can vary according to the market rate for a taxi ride or change in traffic. However, the variance of this function is assumed to be time-invariant. The variance of fare is constrained by customer behavior to refuse above market rate fare and it is assumed that the variance of time to complete a task is constant.

Thus, we define the loss function, $\mathcal{L}(.)$ and the expected value of the loss function as follows: 
\begin{align}\label{Loss}
&\mathcal{L}(Q_t(i,{\pi[i]})) =\frac{1}{2}\bigg(Q_t(i,{\pi[i]})-v_t(i,{\pi[i]})\bigg)^2, \\
&\mathbb{E}\big[ \mathcal{L}(Q_t(i,{\pi[i]})) \big] = \nonumber \\
&\hphantom{aaaa} \frac{1}{2} \bigg\{\bigg(Q_t(i,{\pi[i]})-Q_t^*(i,{\pi[i]})\bigg)^2+\sigma_t^2(i,{\pi[i]}) \bigg\},\nonumber
\end{align}
where $\mathbb{V}\text{ar}[X]=\mathbb{E}[X^2]-\mathbb{E}[X]^2 $ is used in the last equality. By way of a stochastic stability formulation, the Lyapunov function of the system is the expected value of the loss function. 

Adopting the stochastic stability iteration framework from \cite{StochBound}, our optimization problem is to choose $\alpha^*_t(i,\pi[i])$ to minimize the expectation of the Lyapunov function, conditioned on the value at the previous state:
\begin{align}\label{optimization_alpha_def_eqn}
\alpha^*_t(i,{\pi[i]})\!= \operatorname*{arg\,min}_{\alpha_t(\!i,\pi[i])}  \mathbb{E}\bigg[ \mathbb{E}\big[\!\mathcal{L}\big(Q_{t+1}(\!i,\pi[i])\big)\big] | Q_{t}(i,{\pi[i]})\bigg]
\end{align}
\begin{theorem}\label{Theorem-adaptive-alpha}
For SARSA RL~\eqref{Q-update}, the optimal value of $\alpha_t$ for each state-action pair is estimated as a function of exponential moving averages, $f^*$ and $g^*$:
\begin{align}
\alpha^*_t(i,{\pi[i]})&= \frac{ {f^*_{t}(i,{\pi[i]})}^2}{g^*_{t}(i,{\pi[i]})},\label{alpha-final}\\
 f^*_{t+1}(\!i,{\pi[i]})&=  f^*_{t}(\!i,{\pi[i]})\!+\zeta \rho(\frac{\partial  \mathcal{L}}{\partial Q_t(\!i,{\pi[i]})} - f^*_{t}(\!i,{\pi[i]})), \nonumber\\
 g^*_{t+1}(\!i,{\pi[i]})&\!=\!  g^*_{t}(\!i,{\pi[i]})\!+\zeta \rho (\frac{\partial  \mathcal{L}}{\partial Q_t(\!i,{\pi[i]})}^2 \!\!\! - g^*_{t}(\!i,{\pi[i]})), \label{fg} \\
\frac{\partial  \mathcal{L}}{\partial Q_t(i,{\pi[i]})}&=Q_t(i,{\pi[i]}) -\mathcal{R}_{\pi[i]}(i,j)-\gamma Q_t(j,{\pi[j]})\label{gradLSample}
\end{align}
where $\rho=1$, if there is a new update for action ${\pi[i]}$ in state $i$ at time $t$. Otherwise, we set $\rho=0$. Also, $\zeta$ is a design parameter used for exponential convergence, $0 < \zeta < 1$.  We recover \eqref{gradLSample} from the definition of $v_t$ and taking the gradient of \eqref{Loss}. With this formulation, we can compute $\alpha^*_t$ with only $Q$-values and reward update information.
\end{theorem}
\hspace{0.25cm} \begin{proof}
Without loss of generality, we can write the loss function of a state-action pair using a superposition of all the samples as a single state-action pair update: 
\begin{align}\label{Loss-sample}
\mathbb{E}&[\mathcal{L}(Q_t(i,{\pi[i]}))]
\\&=\frac{1}{2}\bigg\{ Q_t(i,{\pi[i]})-\mathcal{R}_{\pi[i]}(i,j)-\gamma Q_t(j,{\pi[j]})\bigg\}^2. \nonumber
\end{align}

To minimize \eqref{Loss}, the Stochastic Gradient Descent (SGD) is used. Using \eqref{Loss-sample}, SGD is obtained as
\begin{align} \label{SGD_update}
Q_{t+1}&(i,{\pi[i]}) = Q_{t}(i,{\pi[i]})-\alpha_t(i,{\pi[i]}) \hspace{0.1cm} \mathbb{E}\big[\frac{\partial   \mathcal{L}(Q_t(i,{\pi[i]}))}{\partial Q_t(i,{\pi[i]})}\big] \nonumber \\
&=  Q_{t}(i,{\pi[i]})-\alpha_t(i,{\pi[i]}) \times\nonumber\\& \hspace{.8cm}\bigg( Q_t(i,{\pi[i]})-\mathcal{R}_{\pi[i]}(i,j)-\gamma Q_t(j,{\pi[j]}) \bigg).
\end{align}

Note this is consistent with SARSA RL update in \eqref{Q-update}. Equation \eqref{SGD_update} is rewritten as
\begin{align} \label{update-iid}
Q_{t+1}&(i,{\pi[i]})= (1-\alpha_t(i,{\pi[i]}) ) Q_{t}(i,{\pi[i]})\nonumber \\
&+\alpha_t(i,{\pi[i]})\bigg(Q_t^*(i,{\pi[i]})+\eta\hphantom{a}\sigma_t(i,{\pi[i]}) \bigg),
\end{align}
where $\eta$ is an i.i.d. sample with a zero-mean and unit-variance Gaussian distribution.
Therefore, using \eqref{Loss} and \eqref{update-iid}, the loss function after one-step SGD update is obtained as
\begin{align*} \label{Expected-Loss}
\mathbb{E}\bigg[ \mathbb{E}&\big[\mathcal{L}\big(Q_{t+1}(i,{\pi[i]})\big)\big] | Q_{t}(i,{\pi[i]})\bigg] \nonumber \\
=&\mathbb{E}\bigg[\frac{1}{2}\bigg( Q_{t}(i,{\pi[i]}) +\alpha_t(i,{\pi[i]})\times \bigg( \\
&\hspace{.8cm}Q_t^*(i,{\pi[i]}) + \eta\hphantom{i}\sigma_t(i,{\pi[i]})-Q_{t}(i,{\pi[i]}) \bigg) \\
&\hspace{.8cm}-Q_t^*(i,{\pi[i]}\bigg)^2+\sigma_{t+1}^2(i,{\pi[i]})\bigg]\nonumber\\
=& \frac{1}{2}\bigg( \big(1-\alpha_t(i,{\pi[i]}) \big)^2 \big( Q_{t}(i,{\pi[i]})-Q_t^*(i,{\pi[i]})\big)^2  \nonumber \\
& \hspace{0.9cm}+\alpha_t^2(i,{\pi[i]})\hphantom{i}\sigma_t^2(i,{\pi[i]})+\sigma_{t+1}^2(i,{\pi[i]}) \bigg) \nonumber
\end{align*}
\hspace{0.25cm} Here, we use Assumption \ref{sigma-timeinvariant} and solve the optimization problem in \eqref{optimization_alpha_def_eqn}:
\begin{equation} \label{alphastar}
\begin{aligned} 
&\alpha^*_t(i,{\pi[i]}) =\operatorname*{arg\,min}_{\alpha_t(i,{\pi[i]})} \bigg\{ \big(1-\alpha_t(i,{\pi[i]}) \big)^2 \times \bigg( \\
& Q_{t}(i,{\pi[i]})-Q_t^*(i,{\pi[i]})\bigg)^2 + (\alpha_t^2(i,{\pi[i]})+1) \sigma_t^2(i,{\pi[i]}) \bigg\} 
\\
&= \operatorname*{arg\,min}_{\alpha_t(i,{\pi[i]})} \bigg\{ -2\alpha_t(i,{\pi[i]})  \big( Q_{t}(i,{\pi[i]})-Q_t^*(i,{\pi[i]})\big)^2 + \\
&\quad \alpha_t^2(i,{\pi[i]}) \bigg( \big( Q_{t}(i,{\pi[i]})-Q_t^*(i,{\pi[i]})\big)^2+ \sigma_t^2(i,{\pi[i]})\bigg) \bigg\} 
\\
&=\frac{\big( Q_{t}(i,{\pi[i]})-Q_t^*(i,{\pi[i]})\big)^2}{ \big( Q_{t}(i,{\pi[i]})-Q_t^*(i,{\pi[i]})\big)^2+\sigma_t^2(i,{\pi[i]})}
\end{aligned}
\end{equation}
In the remainder of this proof, we present a numerical solution to calculate \eqref{alphastar}, at each time-step. Using \eqref{Loss}, the expected value and variance components of $\frac{\partial  \mathcal{L}}{\partial Q_t(i,{\pi[i]})}$ are written as
\begin{align*}
\mathbb{E}\big[ \frac{\partial  \mathcal{L}}{\partial Q_t(i,{\pi[i]})} \big]&= Q_t(i,{\pi[i]})-Q_t^*(i,{\pi[i]}), \nonumber\\
\mathbb{V}\text{ar} \big[ \frac{\partial \mathcal{L}}{\partial Q_t(i,{\pi[i]})} \big]&=\sigma_t^2(i,{\pi[i]}). \nonumber
\end{align*}
This allows us to rewrite $\alpha^*_t(i,{\pi[i]})$ in \eqref{alphastar} as
\begin{align} \label{alphastar2}
\alpha^*_t(i,{\pi[i]})& = {\mathbb{E}\big[ \frac{\partial  \mathcal{L}}{\partial Q_t(i,{\pi[i]})} \big]^2}/{\mathbb{E}\big[ {\frac{\partial  \mathcal{L}}{\partial Q_t(i,{\pi[i]})}}   ^2 \big]}.
\end{align} 
The moving average can be used to calculate the expected values. The exponential moving average with time constant $\zeta$ can be obtained using~\eqref{fg}. By setting
\begin{align*}
f^*_{t}(i,{\pi[i]})= \mathbb{E}\big[ \frac{\partial  \mathcal{L}}{\partial Q_t(i,{\pi[i]})} \big], \ \
g^*_{t}(i,{\pi[i]})= \mathbb{E}\big[ {\frac{\partial  \mathcal{L}}{\partial Q_t(i,{\pi[i]})}}^2 \big],
\end{align*}
The adaptive learning rate \eqref{alphastar2} results in \eqref{alpha-final}. \end{proof}
\begin{remark}
Intuitively, \eqref{alphastar} illustrates that the learning rate is reduced when the measurements (gradients of the SGD) have large covariance. The learning rate will be more affected by the measurements covariance when our estimate Q-value is closer to the optimal Q-value, $Q^*$.
\end{remark}

\subsection{Distributed SARSA RL for Non-stationary Environment} \label{SubSec:Dis-Learning}
In Sec.~\ref{subsec:SARSACentralized} and~\ref{SubSec:adaptive}, we introduced an adaptive SARSA algorithm for non-stationary environments. Here, we present a dynamic average tracking algorithm to estimate the time-varying Q-values in a distributed manner that is, by nature, scalable to a large number of autonomous vehicles. The mathematical overview is as follows. First, we present the proposed update rules with each agent $i$'s structure. Second, we make assumptions on the system to present upper limit bounds on estimate errors. Third, we show convergence of a stochastic difference equation to prove that the estimated Q-values converge to the true values with bounded errors.

The update rule for agent $i$ and the observation pair $(l,a_j^l)$ is proposed as follows:
\begin{align} \label{Dist-Sarsa}
&{\hat{Q}}_{t+1}^i={\hat{Q}}_{t}^i+\sum_{k=1}^{k=N} \mathcal{A}_{ik}(t) \big({\hat{Q}}_{t}^k-{\hat{Q}}_{t}^i\big)+\bold{r}^i_t,  \nonumber\\
&\bold{r}^i_t= N \bold{i}^i_t \alpha^i_t(l,a_j^l) r^i_t,
\\
& {\color{black}r^i_t=\mathcal{R}_{a_j^l}(l,j)+\gamma \hat{Q}_t^i(j,\pi[j])-\hat{Q}_t^i(l,a_j^l)}, \ \ r^i_0=0, \forall i, \nonumber
\end{align}
where $N$ is the number of agents, $\mathcal{A}=[\mathcal{A}_{ik}] \in \mathbb{R}^{N\times N}$ is the adjacency matrix of communication among agents defined in Sec. \ref{Sec:Prelim} and Assumption \ref{conn-graph}, and $j$ is the successor  state after conducting action $a_j^l$ at state $l$. Also, $\bold{i}^i_t$ is a vector with one non-zero entry corresponding to the state-action pair $(l,a^l_j)$, unless agent $i$ does not select an action. Note $\bold{i}^i_t \in \mathbb{R}^{\sum_{k=1}^{n_q} |A(k)|}$, where $n_q$ is the number of cells in the city, and $|A(k)|$ is the cardinality of the action set for a given cell. For example, if every cell has an action to every other cell, $\bold{i}^i_t \in \mathbb{R}^{n_q^2}$. The learning correction for a state-observation pair is $r_t^i$, and $\bold{r}^i_t$ is the stacked vector form.

The agent's structure is as follows. Each agent $i$ maintains its estimate of Q-values for state action pairs at time $t$, in vector ${\hat{Q}}^i_t \in\mathbb{R}^{\sum_{k=1}^{n_q} |A(k)|}$. The agent $i$'s estimate of the optimal learning-rate vector is obtained as:
\begin{align} \label{Dist-alpha}
&\alpha^i_t(j,a_j^l)=\frac{\hat{ {f}}^i_t(j,a_j^l)^2}{\hat{ {g}}^i_t(j,a_j^l)}, \   
\hat{ {f}}^i_t=\hat{ {f}}^i_{t-1}+\zeta ({\omega}^i_{t-1}-\hat{ {f}}^i_{t-1}), \\
&\hat{ {g}}^i_t=\hat{ {g}}^i_{t-1}+\zeta ((\omega^i_{t-1})^2-\hat{ {g}}^i_{t-1}), \nonumber\\
&{\omega}^i_{t}= {\omega}^i_{t-1}+\sum_{k=1}^{k=N} \mathcal{A}_{ik}(t)({\omega}^i_{t-1}-{\omega}^k_{t-1})+N (\bold{i}^i_t r^i_t-\bold{i}^i_{t-1} r^i_{t-1})\nonumber
\end{align}
where $\hat{ {f}}^i_t$ and $\hat{ {g}}^i_t$ are the agent $i$'s estimates of $f^*_t$, and $g^*_t,$ defined in \eqref{fg}, respectively. Also, ${\omega}^i_{t}$ is the estimate of the gradient of the loss function, $\frac{\partial  \mathcal{L}}{\partial Q_t(j,a_j^l)}$, written in a vector form and $(\omega^i_{t-1})^2$ is obtained by squaring each element. The initial values are chosen as ${\omega}^i_0=\bold{0}$, and $\hat{ {f}}^i_0=\hat{ {g}}^i_0=\bold{1}, \forall i$. The information updates available to agents are local customer requests data: state transitions of departure and arrival cells, fare, and travel time. The algorithms in  \eqref{Dist-Sarsa} and \eqref{Dist-alpha} are the distributed forms of \eqref{Q-update} and \eqref{alpha-final}, respectively. Using \eqref{Dist-Sarsa} and \eqref{Dist-alpha}, each agent only requires local information and local interactions to update its values.
\newline \indent The agents share their Q-value and $\omega_t$ estimates with their neighbors. We make the following assumptions of our system. 
\begin{assumption}\label{bound-changes}
There exists a bounded, time-invariant constant, $r_{\max}$, such that for all agents and all time, $\abs{r^i_{t}} \leq r_{\max}$. This assumption also implies there exists another constant, $\Delta r_{\max}$, where $\abs{r^i_{t} - r^i_{t-1}} \leq \Delta r_{\max}$ for all agents and all time.  
\end{assumption}
\begin{assumption}\label{conn-graph}
The digraph $\mathcal{G}_t \triangleq (V,E_t)$, with its adjacency matrix $\mathcal{A}(t)=[\mathcal{A}_{ik}(t)]$ from (\ref{Dist-Sarsa}),
satisfy the following:
\begin{itemize}
\item[(I)] Periodic Strong Connectivity:
There exists a positive integer $\mathbf{b}$, such that the digraph $\mathcal{G}_t \triangleq (V,E_t\cup E_{t+1} \cup \cdots \cup E_{t+\mathbf{b}-1})$ is strongly connected for all $t$.
\item[(II)] Non-degeneracy: There exists a constant $\gamma \in (0,1)$ such that $\mathcal{A}_{ik}(t) \in \{0\} \cup [\gamma,1]$.
\item [(III)] Balanced Communication: The matrix $\mathcal{A}(t)$ is doubly stochastic for all $t$, i.e., $\mathbf{1}^T\mathcal{A}(t)= \mathbf{1}^T$ and $\mathcal{A}(t)\mathbf{1}= \mathbf{1}.$
\end{itemize}
\end{assumption}
\hspace{0.25cm} Now, we present Theorem~\ref{corollary-DAT} and Corollary~\ref{remark-alpha-bound} to define the following terms: upper limit of estimation errors, $\delta_Q$ and upper limit of estimation error of learning-rate, $\delta_a$. 
\begin{theorem} \label{corollary-DAT}
Suppose that Assumptions \ref{bound-changes} and \ref{conn-graph}  hold. Under the control laws given by \eqref{Dist-Sarsa} and \eqref{Dist-alpha}, the distributed average tracking goals for all agents ($\forall i$) are achieved in some finite time $\kappa$ with bounded error $\delta_Q$ and $\delta_\omega$, i.e., 
\begin{align*}
\lim_{t\rightarrow \infty}\norm{\hat{Q}_t^i-\frac{1}{N}\big( \sum_{j=1}^{j=N}\hat{{Q}}_{0}^j+\sum_{\tau=1}^{\tau=t-1} \sum_{j=1}^{j=N} \bold{r}^j_\tau \big)} \leq \delta_Q, \\
\lim_{t\rightarrow \infty}\norm{\omega^i_{t}-\rho\frac{\partial \mathcal{L}}{\partial Q_t(i,{\pi[i]})}}\leq \delta_\omega, 
\end{align*}
where 
\begin{align}
\delta_Q &= \frac{2\sqrt{N} r_{\max} }{1-\max_{\mathcal{A}(t)}\sigma_{N-1}(\mathcal{A}(t))} \\
\delta_\omega &= \frac{2\sqrt{N} \Delta r_{\max} }{1-\max_{\mathcal{A}(t)}\sigma_{N-1}(\mathcal{A}(t))},
\end{align}
where $\sigma_{N-1}(\mathcal{A}(t))$ denotes the second largest singular value of the matrix $\mathcal{A}(t)$ and $N$ term denotes the number of agents.
\end{theorem}
\begin{proof}
Both \eqref{Dist-Sarsa} and \eqref{Dist-alpha} are distributed dynamic average tracking equations for discrete time signals. Convergence analysis of equations of this form are presented in \cite{DBF}. Here, our goal is to show that the estimated signals $\hat{Q}_t^i$ and $\omega_t^i$ converge to the average of all agents signals with bounded error. 
The Distributed Bayesian Filtering algorithm (DBF) presents an estimation error of the exponentially-stabilizing consensus estimation algorithm (see Corollary 6 of \cite{DBF}). By using the discrete Gronwall lemma, its error bound, $\delta$ can be manipulated as follows:
\begin{align*}
\delta \leq \frac{2\sqrt{N}\Delta}{1-\max_{\mathcal{A}(t)}\sigma_{N-1}(\mathcal{A}(t))}, \hspace{0.5cm} \forall t>\kappa 
\end{align*}
where $\Delta$ is an upper bound of the update value, corresponding to $r_{\max}$ and $\Delta r_{\max}$ for $Q^i_t$ and $\omega^i_t$ signals, respectively. 
\end{proof}
\begin{corollary} \label{remark-alpha-bound}
With proper tuning, the upper limit of the estimation error of the optimal learning rate $\delta_\alpha$ is bound by one for each state action pair.
\begin{align*}
\abs{\alpha^i_t-\alpha^*_t} \leq \delta_\alpha \leq 1
\end{align*}
\end{corollary}
\begin{proof} 
The error bound of the estimate of the optimal learning-rate for each state-action pair can be obtained as
\begin{align} \label{alpha-bound-error}
\abs{\alpha^i_t-\alpha^*_t}=\abs{\frac{(\hat{f}^i_t)^2}{\hat{g}^i_t}-\frac{({f}^*_t)^2}{{g}^*_t}}\leq \abs{\frac{\delta_{\hat{f}}^2+2\delta_{\hat{f}} f_t^*}{{g}^*_t({g}^*_t-\delta_{\hat{g}})}} \leq \delta_\alpha,   \forall t>\kappa,
\end{align}
where $\delta_{\hat{f}}$ and $\delta_{\hat{g}}$ are the estimation errors, i.e., $\|f^*_t-\hat{f}^i_t\|\leq \delta_{\hat{f}}$ and $\|g^*_t-\hat{g}^i_t\|\leq \delta_{\hat{g}}, \forall t$ and for all state-action pairs. Note that the state-action argument for true and estimate values of $\alpha_t(i,\pi[i])$, $f_t(i,\pi[i])$ and $g_t(i,\pi[i])$ are dropped for readability. Here, it is assumed that the environment changes, defined in Assumption~\ref{bound-changes}, are slow enough to have ${g}^*_t> \delta_{\hat{g}}, \forall t$. From \eqref{alphastar}, we know that $\alpha^* \leq 1$, it is easy to see that by selecting a proper scaling for reward $\mathcal{R}_{a_l}(\cdot)$, the upper bound error $\delta_\alpha$ can remain lower than $1$. Scaling the reward function with a positive constant scales all Q-values for all state-action pairs, however $\delta_\alpha$ will be scaled down due to the effect of having ${g^*_t}^2$ in its denominator. 
\end{proof}

Now we present Theorem~\ref{corollary_sde} to show the equivalence of our update law to a particular stochastic difference equation.
\begin{theorem}\label{corollary_sde}
Suppose that Assumptions \ref{sigma-timeinvariant}, \ref{bound-changes}, and \ref{conn-graph} hold. Under the control laws given by \eqref{Dist-Sarsa} and \eqref{Dist-alpha}, the governing difference equation for agent $i$ is as follows: 
\begin{align}
\hat{Q}_{t+1}^i(i,\pi[i])&=\hat{Q}_{t}^i(i,\pi[i])+\label{eq:theorem2} \\
\alpha^k_t &\bigg[\mathcal{R}_{\pi[i]}(i,j)+\gamma \hat{Q}^k_t(j,\pi[j])-\hat{Q}_t^k(i,\pi[i])\bigg]+\epsilon^i\nonumber
\end{align}
\end{theorem} where $|\epsilon^i| \leq \delta_Q$ is a stochastic estimation error and $\delta_Q$ is defined in Theorem~\ref{corollary-DAT}.

\begin{proof}
Under Assumptions \ref{bound-changes} and \ref{conn-graph}, and by using the control laws \eqref{Dist-Sarsa} and \eqref{Dist-alpha}, Theorem~\ref{corollary-DAT} holds. Then, by selecting $\hat{Q}_{0}^i = Q^*_{\text{MDP}}, \forall i$ and defining $\bold{\varepsilon}^i$ as the aggregated error vector of scalar values $\epsilon^i$, we have 
\begin{align}\label{eq-intheorem} 
&\hat{Q}_{t+1}^i=\hat{Q}_{0}^i+\frac{1}{N}\sum_{\tau=1}^{\tau=t} \sum_{j=1}^{j=N} \bold{r}^j_\tau+\bold{\varepsilon}^i\nonumber\\
&\hspace{0.5cm}=\hat{Q}_{0}^i+\frac{1}{N}\sum_{\tau=1}^{\tau=t-1} \sum_{j=1}^{j=N} \bold{r}^j_\tau + \frac{1}{N} \sum_{j=1}^{j=N} \bold{r}^j_t\!+\!\bold{\varepsilon}^i\\
&\hspace{0.5cm}=\hat{Q}_{t}^i+ \sum_{j=1}^{j=N}\bold{i}^j_t \alpha^j_t r^j_t+\bold{\varepsilon}^i\nonumber
\end{align}
where in last equality we have used the fact that $\hat{Q}_{t}^i=\hat{Q}_{0}^i+\frac{1}{N}\sum_{\tau=1}^{\tau=t-1} \sum_{j=1}^{j=N} \bold{r}^j_\tau$. This is trivially satisfied by the original $Q_t$ update equation (\ref{Dist-Sarsa}).

Now, let's rewrite \eqref{eq-intheorem} for a specific state-action pair $(i,\pi[i])$. Without loss of generality and for simplicity, assume that only one of the agents, $k$, observes a new update for this pair at time $t$. Then, by letting $\epsilon^i$ denote the element value of $\bold{\varepsilon}^i$,  we have: 
\begin{align}
&\hat{Q}_{t+1}^i(i,\pi[i]) =\hat{Q}_{t}^i(i,\pi[i])+ \sum_{j=1}^{j=N} \alpha^j_t r^j_t+\epsilon^i 
\end{align}
which results in (\ref{eq:theorem2}) after substituting the definition of $r^j_t$.
\end{proof}
\hspace{0.25cm} Next, we present an existing stochastic stability result from \cite{StochBound} that will be used.
\begin{theorem}\label{Theo-StochasticLyap} 
Suppose that $x_t$ is generated by
\begin{align}
x_{t+1}=h(x_t)+l(x_t)w_t,
\end{align}
where $x_t \in \mathbb{R}^n$, and $w_t \in \mathbb{R}^m$ is a sequence of uncorrelated normalized Gaussian random variables. If there exists a function $\mathbb{V}(\Omega_t)$ that satisfies
\begin{itemize}
\item[1)] For a positive $c$, we have $\mathbb{V}(\Omega_t) \geq c\norm{\Omega_t}, \forall \Omega_t.$
\item[2)] $\mathbb{E}\big[\mathbb{V}(\Omega_{t+1}) |\mathbb{V}(\Omega_t)\big]-\mathbb{V}(\Omega_{t}) \leq k'-k''\mathbb{V}(\Omega_{t}), \forall \Omega_{t}, k'>0,$ and $0 <k''\leq 1$.
\end{itemize}
Then, $\mathbb{E}[\norm{\Omega_{\infty}}] \leq \frac{k'}{c k''}$.
\end{theorem}
\hspace{0.25cm} Finally, we use all these results to present a novel contribution, Theorem \ref{proposition-bounded-optimality} convergence of estimation error to a bound. 
\begin{theorem} \label{proposition-bounded-optimality}
Suppose that we have a stochastic difference equation
\begin{align} \label{Q-Updates-withError}
Q^i_{t+1}&(i,\pi[i])= Q^i_t(i,\pi[i]) +[\alpha^*(t)+\epsilon^k_\alpha]\\
& \times \big(\mathcal{R}_{\pi[i]}(i,j)+\gamma Q^k_t(j,\pi[j])-Q^k_t(i,\pi[i])\big)+\epsilon^i,\nonumber
\end{align}
where $|\epsilon^k_\alpha|\leq \delta_\alpha$ and $|\epsilon^i|\leq \delta_Q$ are stochastic scalars with known upper bound. Then, we have 
\begin{align*}
\lim_{t\to \infty}\mathbb{E}[\norm{Q^i_{t}(i,\pi[i])-Q^*_{t}(i,\pi[i])}]\leq \Delta,
\end{align*}
where $Q^*_{t}(\cdot)$ is the aggregated vector of the optimal $Q$-values at time $t$.
\end{theorem}
\begin{proof}
Here, using Theorem \ref{Theo-StochasticLyap}, we will show that $\mathbb{E}[\Omega_\infty]$ remains bounded, where $\Omega^i_t=Q^i_t(i,{\pi[i]})-Q_t^*(i,{\pi[i]})$. 
We define $\mathbb{V}(\Omega^i_t)=\mathbb{E}\big[ \mathcal{L}(Q^i_t(i,{\pi[i]}))]\geq c\norm{\Omega^i_t}, \forall \Omega^i_t$ and for a positive constant $c>0$. Hence, the first condition in Theorem \ref{Theo-StochasticLyap} is satisfied. Now, using \eqref{Q-Updates-withError} and \eqref{Loss}, we rewrite $\mathbb{E}\big[\mathbb{V}(\Omega^i_{t+1}) |\mathbb{V}(\Omega^i_t)\big]$ as
\begin{align*}
&\mathbb{E}\big[\mathbb{V}(\Omega^i_{t+1}) |\mathbb{V}(\Omega^i_t)\big]=\mathbb{E}\bigg[\mathbb{E}\big[ \mathcal{L}\big(Q^i_{t+1}(i,{\pi[i]})\big)] | Q^i_{t}(i,{\pi[i]}))\bigg] 
\\ 
&= \frac{1}{2}\bigg[ \big( {\Omega^i_t}- [\alpha^*_t+\epsilon^k_\alpha] {\Omega^k_t} \big)^2
+{[\alpha^*_t+\epsilon^k_\alpha]}^2 {\sigma_t}^2(i,{\pi[i]})\nonumber\\
&+{\sigma_{t+1}}^2(i,{\pi[i]})+({\epsilon^i})^2 +2{\epsilon^i}\big(\Omega^i_t-[\alpha^*_t+\epsilon^k_\alpha]\Omega^k_t  \big) \bigg]
\end{align*}
We define $\epsilon^c_t = Q^i_t(i,{\pi[i]})-Q^k_t(i,{\pi[i]})$, equivalent to $\epsilon^c_t = \Omega^i_t-\Omega^k_t$. Note that the result in Theorem~\ref{corollary-DAT} holds for all agents; hence, $\abs{\epsilon^c_t} \leq \delta_Q$. Now, setting $\Omega^i_t = \Omega^k_t + \epsilon^c_t$ results in:
\begin{align*}\small
&\mathbb{E}\big[\mathbb{V}(\Omega^i_{t+1}) |\mathbb{V}(\Omega^i_t)\big] =\\
&\hspace{0.5cm}\frac{1}{2}\bigg[ \big( [1- \alpha^*_t-\epsilon^k_\alpha] {\Omega^k_t} +\epsilon^c_t\big)^2 +{[\alpha^*_t+\epsilon^k_\alpha]}^2 {\sigma_t}^2(i,{\pi[i]})\nonumber\\
&\hspace{0.5cm}+{\sigma_{t+1}}^2(i,{\pi[i]})+({\epsilon^i})^2 +2{\epsilon^i}\big(\Omega^k_t+\epsilon^c_t-[\alpha^*_t+\epsilon^k_\alpha]\Omega^k_t  \big) \bigg]
\end{align*}
Now, by using \eqref{alphastar}, we can set $\alpha_t^* = \frac{(\Omega^k_t)^2}{(\Omega^k_t)^2 + \sigma_t^2}$:
\begin{align*}
&\mathbb{E}\big[\mathbb{V}(\Omega^i_{t+1}) |\mathbb{V}(\Omega^i_t)\big] =\\
&\hspace{0.5cm}\frac{1}{2}\bigg[\frac{1}{(\Omega^k_t)^2 + \sigma_t^2} \big( (\Omega^k_t)^2(\epsilon^c_t + \epsilon^i - \epsilon^k_{\alpha}\Omega^k_t)^2) + \\ 
&\hspace{0.5cm}[(\epsilon^c_t + \epsilon^i)^2 - 2(\epsilon^c_t + \epsilon^i)(\epsilon^k_{\alpha} - 1)\Omega^k_t + \\
&\hspace{0.5cm}(1+2(\epsilon^k_{\alpha})^2) (\Omega^k_t)^2]\sigma_t^2 + (\epsilon^k_{\alpha})^2 \sigma_t^4 \big) + \sigma_{t+1}^2 \bigg]
\end{align*}
To have the second condition in Theorem \ref{Theo-StochasticLyap}, we need to show that there exist two constants $k''>0,$ and $0 <k'\leq 1$, such that $\mathbb{E}\big[\mathbb{V}(\Omega^i_{t+1}) |\mathbb{V}(\Omega^i_t)\big]-\mathbb{V}(\Omega_{t}) +k'\mathbb{V}(\Omega_{t})-k''\leq 0, \forall \Omega_{t}$. Now, we use Assumption \ref{sigma-timeinvariant} and group terms of $\Omega^k_t$:
\begin{align*}
&\mathbb{E}\big[\mathbb{V}(\Omega^i_{t+1}) |\mathbb{V}(\Omega^i_t)\big]+(k'-1)\mathbb{V}(\Omega^i_{t})  -k'' = \\
&\hspace{0.5cm}(\Omega^k_t)^4[(\epsilon^k_\alpha)^2 + k' - 1] + \\
&\hspace{0.5cm}(\Omega^k_t)^3[2\epsilon^c_t - 2\epsilon^c_t\epsilon^k_{\alpha} - 2\epsilon^i\epsilon^k_\alpha + 2\epsilon^c_tk'] + \\
&\hspace{0.5cm}(\Omega^k_t)^2[(2\epsilon^c_t\epsilon^i + (\epsilon^i)^2 + (\epsilon^c_t)^2k' - 2k'')+\sigma_t^2(2(\epsilon^k_\alpha)^2+2k')] + \\
&\hspace{0.5cm}(\Omega^k_t)^1[\sigma_t^2(2\epsilon^i - 2\epsilon^c_t\epsilon^k_\alpha - 2\epsilon^i\epsilon^k_\alpha + 2\epsilon^c_tk')] + \\
&\hspace{0.5cm}(\Omega^k_t)^0[\sigma_t^2(2\epsilon^c_t\epsilon^i+(\epsilon^i)^2+(\epsilon^c_t)^2k' - 2k'') + \sigma_t^4((\epsilon^k_\alpha)^2+k')]
\end{align*}
Now, by applying known bounds to error terms, $\epsilon$, we can find an upper bound for each of the coefficients. It is easy to observe that a negative coefficient on the leading term will lead to a negative result at large $\Omega_k^t$ and, if the zeroth order term can be made negative with arbitrarily large magnitude, the zeroth order term can shift the polynomial negative such that the expression remains negative during the transient response of the first, second, and third order terms before the fourth order dominates. We consider the upper bound of the fourth term: 
\begin{align*}
[(\epsilon^k_\alpha)^2 + k' - 1] \leq [\delta_a^2 + k'-1] 
\end{align*}
With Corollary~\ref{remark-alpha-bound} that $0 < \delta_a < 1$, it is easy to see that $k'$ can be selected such that 4th order coefficient is always negative and $0 < k' \leq 1$. Now, we consider the zeroth order term upper bound:  
\begin{align*}
[\sigma_t^2(2\epsilon^c_t\epsilon^i+(\epsilon^i)^2+(\epsilon^c_t)^2k' - 2k'') + \sigma_t^4((\epsilon^k_\alpha)^2+k')] \leq \\
[\sigma_t^2(3\delta_Q^2+(\delta_Q)^2k') + \sigma_t^4((\delta_a)^2+k')- 2\sigma_t^2k'']
\end{align*}
It is easy to see that an arbitrarily large $k''$ drives the upper bound of the zeroth term arbitrarily largely negative such that $k'' > 0$. Therefore, the conditions in Theorem \ref{Theo-StochasticLyap} are satisfied, and $\mathbb{E}[\norm{\Omega_{\infty}}] \leq \frac{2k''}{k'}$.
Thus, we have $\lim_{t\to \infty}\mathbb{E}[\norm{Q^i_{t}(i,\pi[i])-Q^*_{t}(i,\pi[i])}]\leq \Delta$, which completes the proof.
\end{proof}
\section{Distributed Local Task Assignment} \label{Sec:game}
	Once the agents have a Q-value table of optimal policies, the agents must coordinate to assign tasks uniquely in order to maximize the profit of a company. We propose a distributed method using a potential game and binary log-linear learning, as shown in green in the right side of Fig.~\ref{Schematic}. We use a distributed framework game-based method that is compatible the distributed SARSA RL estimation of optimal policies presented in the previous section. There are various algorithms in the literature to solve the task assignment problem. The well-known Hungarian method \cite{kuhn, Edmonds}, auction based methods \cite{Bertsekas1988}, and parallel algorithms \cite{Bertsekas}, and their applications in multi-robot target and task assignment \cite{Schumacher, Jin, Bellingham} can be employed to solve our problem formulation. However, these algorithms are mainly designed to solve the assignment problem in a centralized manner. In our framework, number of customer requests, $m_c$, and number of agents, $N$, can be large numbers; hence, it might not be feasible to pass all information at each time step to a command center that could process the information. Furthermore, the complexity of the overall system makes the problem of constructing a centralized optimal policy computationally heavy or even intractable. Some decentralized methods have been introduced in literature to tackle this problem \cite{morgan2016swarm,Dionnedec, sujitdec, How2009}. In \cite{morgan2016swarm}, a distributed auction-based algorithm is introduced, where the task assignment problem is solved in a distributed manner. In~\cite{How2009}, the consensus algorithm is employed to find the centralized solution in a distributed manner. However, by using this approach, the size of the problem is not reduced, and only the requirement for having a central node is relaxed. As a result, the algorithm for a large number of customers and agents becomes intractable.    
\subsection{Game Design} \label{game-design}
	In this section, we present a game-based local interaction among agents to select their customers in a distributed manner. In particular, we consider a problem with $m_c$ customer requests and $N$ agents available. The agents can only see requests and other agents, if they are in a range. 
\begin{definition}\label{Tasks}
The pick-up and delivery task is denoted by $T(c^j_p, c^l_d)$. The task is completed when the agent picks up the customer from the pick-up point $c^j_p \in j$ and delivers them to destination $c^l_d \in  l$. Note that the terms action and task are used interchangeably in this section.
\end{definition}
\begin{assumption} \label{AvailableTasks}
Each agent is aware of any pick-up requests within radius, $r_c$, from its current position, $p^i_t$. In other words, the tasks available for agent $i$ are denoted by the set $U^i_t=\{T(c^j_p, c^l_d)\mid  \ \norm{p^i_t-c^j_p}<r_c\}$. 
\end{assumption}
\begin{assumption} \label{Communication}
Each agent is able to communicate with its neighbors to exchange information. The set of neighbors of agent $i$ is given by $N_i^{\text{comm}} (t):=\{j|\norm{p^i_t-p^j_t}_2 \leq R_i^\text{comm}\},$ where $R_i^\text{comm}$ is the communication range of agent $i$.
\end{assumption}

We require $R_i^\text{comm}$ to be larger than or equal to $2r_c$. That is, when the agents have an action set intersection, they can communicate with each other. The agent's action at time step $t$ is denoted by $u^i_t$, where $u^i_t\in U^i_t$ and $U^i_t$ is the available action set for agent $i$ defined in Assumption \ref{AvailableTasks}. The action profile of all agents is  denoted by $u_t =(u^1_t,...,u^N_t)\in U_t:=\prod_{i=1}^N U^i_t$.

Now, we propose a non-cooperative game to solve the task assignment problem in a distributed fashion. First, we design a potential game. To formulate our task assignment problem as a game, we design a utility function, $J^i$, that aims to capture an action's marginal contribution on the company's profit, for each agent $i$. 
\begin{align}
J^i (u_t)&=H( \{u^i_t\} \cup u^{-i}_t)- H( \{u^i_0\} \cup u^{-i}_t)  \label{1} \\
 H(u_t)&=\sum_{i=1}^N h(u^i_t),\\
 h(u^i_t)&=Q(j,a_l)+\mathcal{R}_{l}(j,l)-C\norm{p^i_t-c^j_p}, \label{futilty}
\end{align} for $u^i_t=T(c^j_p, c^l_d)$, where $c^j_p \in j$ and $c^l_d \in l$. Also, $u^i_0$ is the null action of agent $i$, and $u^{-i}_t$ denotes the actions of all agents other than agent $i$. $C\norm{p^i_t-c^j_p}$ is the cost for agent $i$ moving from vehicles  current position, $p^i_t$ to the pick-up location $c^j_p$, where $C$ is a design constant. 

The marginal contribution, $J^i$, is sometimes referred to as the Wonderful Life Utility (WLU)\cite{wonderful}.
Note that the utility function $J^i$ is local for agent $i$ over the region defined in Assumption \ref{AvailableTasks}. Note that $J^i$ is dependent only on the actions of $\{i\}\cup N_i^{\text{sense}}(t)$, where $N_i^{\text{sense}}(t)=\{j\mid U^i_t \cap U^j_t \neq \emptyset \} $. This means that agent $i$ can calculate $J^i$ while only  knowing the actions of the agents whose available action sets have an intersection with that of agent $i$. As mentioned before, by setting $R^\text{comm}\geq 2r_c$, if $j \in N_i^{\text{sense}}(t)$, it follows that $j \in N_i^\text{comm} (t)$.
That is, when the agents have an action set intersection, they can communicate with each other.
Hence, there is no need for an agent to know the actions of the agents that do not have an action set intersection with it. This makes the defined utility function local. 
Our potential game is defined as follows.
\begin{lemma} \label{lem1}
The assignment game $ \Upsilon := \left\langle N, U, J \right\rangle $, where $J= \{J^i, i=1,..., N\}$ with $J^i$ given by \eqref{1}, is a potential game with the potential function $\Phi$:
\begin{align}\label{phi1} \small
\begin{split}
 \Phi(u_t)&=H(u_t). \normalsize
\end{split}
\end{align}
\end{lemma}
\begin{proof}
A potential game has to satisfy
\begin{equation}\label{proooof0}
\begin{split}
 \Phi \big( {u'}^i_t,& u^{-i}_t\big) - \Phi \big(u_t \big)   = J^i\big({u'}^i_t,u^{-i}_t\big)- J^i \big(u_t\big)
\end{split}
\end{equation}
for any agent $i =1,..., N$ and action ${u'}^i_t \in U^i_t$.
According to (\ref{1}) and (\ref{phi1}), it is easy to see that (\ref{proooof0}) holds.
\end{proof}
\subsection{Game Theory Extension: Ridepooling Utility} \label{SubSec:ridesharing}
	Here, we design a new utility function for our task assignment game, where ridepooling is also considered.
	It is assumed that the vehicle can only service two customers at the same time. This assumption holds for both UberPool and LyftLine, where ridepooling option is offered to customers.
    Instead of using Assumption \ref{AvailableTasks}, the available tasks for each agent is defined with Assumption \ref{AvailableTasks-RideShare}.
\begin{assumption} \label{AvailableTasks-RideShare}
Available tasks for agent $i$ are denoted by:
 $\bar{U}^i_t=\{[(T(c^j_p, c^l_d), T(c^{j'}_p, c^{l'}_d)] \mid \forall T(c^j_p, c^l_d), T(c^{j'}_p, c^{l'}_d) \in U^i_t\}$.
\end{assumption}
\hspace{0.25cm} The set of available tasks in Assumption \ref{AvailableTasks-RideShare} contains all possible coupled customers tasks, including having only one customer. Now, we define the utility function, $h(u^i_t), u^i_t\in \bar{U}^i_t$, for ride-pooling as:	
\begin{align} \label{futiltyshared}
& h(u^i_t)= \exp^{-C' \beta} \nonumber\bigg[ Q(k,a_{\mathsf{d}_2})+\mathcal{R}_{l} (j,l) \\&
\hspace{3cm}+\mathcal{R}_{l'} (j',l') -C\norm{p^i_t-c^{\mathsf{p}_1}_p}\bigg],\nonumber \\&   \mathsf{p}_1=\arg \min_{o}\norm{p^i_t-c^o_p}, \ o\in \{j,j'\},\mathsf{p}_2= \{j,j'\}\backslash \{\mathsf{p}_1\}, \nonumber\\
&\mathsf{d}_1= \arg \min_{o} \norm{c^{\mathsf{p}_2}_p -c^o_d}, \ o \in \{l,l'\}, \mathsf{d}_2= \{l,l'\}\backslash \{\mathsf{d}_1\}, \nonumber \\
&\text{Path}_{\min}=\norm{p^i_t-c^{\mathsf{p}_1}_p}+ \norm{c^{\mathsf{p}_1}_p-c^{\mathsf{p}_2}_p} \nonumber
 \\& \hspace{3cm}+\norm{c^{\mathsf{p}_2}_p-c^{\mathsf{d}_1}_d}+\norm{c^{\mathsf{d}_1}_d-c^{\mathsf{d}_2}_d}, \nonumber\\
& \beta=\frac{\text{Path}_{\min}}{\min\{ \norm{c^j_p-c^{l}_d},\norm{c^{j'}_p-c^{l'}_d}\}}-1,
\end{align}
where $\mathsf{p}_1$ and $\mathsf{p}_2$ are the first and second pick-up location indices, respectively. The first and second drop off indexes are denoted by $\mathsf{d}_1$ and $\mathsf{d}_2$, respectively. $\text{Path}_{\min}$ denotes the shortest path that agent must travel to accomplish its task, picking up and dropping off both customers, computed from the agent is current position. The cost of sharing the ride for two customers $T(c^j_p, c^l_d)$ and $T(c^{j'}_p, c^{l'}_d)$ is denoted by design parameter, $beta$, where $\beta\geq 0$ and a smaller $\beta$ indicates a better coupling. $C'>0$ is a positive constant to be selected, and it is assumed that $p^i_t \in k$. Now, replacing \eqref{futilty} with \eqref{futiltyshared}, a new utility function for ride-pooling can be calculated. The same reasoning holds, and it is easy to see that the game remains a potential game.

\subsection{Nash Equilibrium Convergence Using Binary Log-Linear Learning} \label{subsec:learning}
We need a distributed adaptation rule to converge to a Nash equilibrium defined in Sec. \ref{game-design}. The goal is that each agent can maximize its own utility function using these rules. Game theoretic reinforcement learning provides iterative algorithms to reach a Nash equilibrium\cite{robust5},\cite{robust7}.

Binary log-linear learning is a modified version of the log-linear learning for potential games, where only a single player updates its action at each iteration. The agents are allowed to explore and can select non-optimal actions but with relatively low probabilities. This plays an important role for agents to escape the suboptimal actions, and as a result the probability of finding a better Nash equilibrium is increased.
Binary log-linear learning can be used for varying available action sets. In \cite{revisiting}, it is shown that a potential game will converge to stochastically stable actions, where these actions are the set of potential maximizers if the feasibility and reversibility assumptions are satisfied on the agents available sets.
Binary log-linear learning is defined in our system as follows: At each time $t$, one agent $i$ is randomly selected and allowed to alter its current action, $u^i_t$, while all other agents repeat their actions, i.e., $u^{-i}_t = u^{-i}_{t - 1}$. The selected agent $i$ chooses a trial action $u'^i_t$ uniformly randomly from the available action set $U^i_t$. The player calculates, $J^i(u'^i_t, u^{-i}_{t - 1})$, the utility function for this trial action. Then agent $i$ changes its action according to the following distribution:
\begin{align} \label{log}
 P_i^{(u^i_{t - 1}, u^{-i}_{t - 1})} (t) &=\frac {\exp({\frac{1}{\tau}J^i(u_{t - 1})})}{{\exp({\frac{1}{\tau}J^i(u_{t - 1})})}+ {\exp({\frac{1}{\tau}J^i(u'^i_t, u^{-i}_{t - 1})}})},\nonumber \\
 P_i^{(u'^i_t, u^{-i}_{t - 1})} (t)&= \frac{\exp({\frac{1}{\tau}J^i(u'^i_t,u^{-i}_{t - 1})})}{{\exp({\frac{1}{\tau}J^i(u_{t - 1})})}+{\exp({\frac{1}{\tau}J^i(u'^i_t, u^{-i}_{t - 1})}})},\nonumber \\
\end{align}
where $P_i^{(u^{i}_{t - 1}, u^{-i}_{t - 1})} (t)$ denotes the probability of choosing action $u^i$ at time $t$ while other agents are repeating their action $ u^{-i}_{t - 1}$. Note that the sum of $P_i^{(u^i_{t - 1}, u^{-i}_{t - 1})} (t)$ and $P_i^{(u'^i_t, u^{-i}_{t - 1})}$ is 1, so the probability of taking an action other than $u^i_{t - 1}$ or $u'^i_t$ is zero. The coefficient $\tau > 0$ is a design parameter specifying how likely agent $i$ chooses a suboptimal action, to specify the trade-off between exploration and exploitation. For $\tau \to \infty$, (exploration) the learning algorithm will choose the action $u^i_{t - 1}$ or $u'^i_t$ with an equal probability while for $\tau \to 0$, (optimality) it will choose the action which has the greatest utility function among the set $u^i_{t - 1}$ and $u'^i_t$.

\section{Simulation and Discussion} \label{Sec:sim}
To validate the proposed algorithms, we prepare simulations using taxi data provided by the city of Chicago, \cite{chicagoData}. The city is partitioned into $77$ cells (as shown in Fig.~\ref{chicago}). In provided data, each entry contains the pick-up and drop-off cells of the trip, time-stamp, duration and fare. We use data from May $2017$, which gives us approximately one million trips to analyze. In order to numerically compare the centralized and distributed SARSA RL algorithms, we use the respective learning update with the same game theory task assignment. First, we illustrate the algorithm at different time-steps, then we show convergence of our distributed SARSA RL algorithm to the centralized solution with two metrics: tracking estimated Q values and comparing total revenue of each algorithm. Finally, we show the economic advantage of learning a time-varying environment. Also, we provide an animation of the task assignment, provided in \url{https://youtu.be/T9DwK8-W6xI}.

The simulation procedure is shown in Algorithm 1. We choose the following standard design parameters: moving average constant $\zeta = 0.2$, exploration/exploitation constant $\tau = 0.5$, and cost to travel constant $C = 20$. 

\begin{figure}[t]
\begin{center} \hspace*{0cm}
\vspace{0.1cm} {\scalebox{0.6}{\includegraphics*{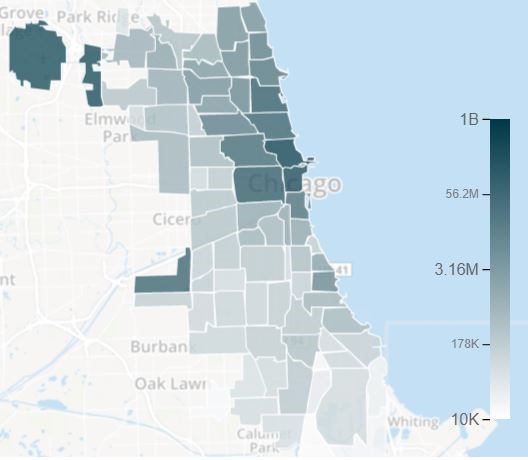}}}
\vspace{0cm} \caption{Grid map of city of Chicago with 77 cells . Color map corresponds to number of taxi trips.} \vspace{-0.3cm}
\label{chicago}\end{center}
\end{figure}

\begin{figure*}[t]
\begin{center} \hspace*{-0.05cm}
\vspace{0.1cm} {\scalebox{0.52}{\includegraphics*{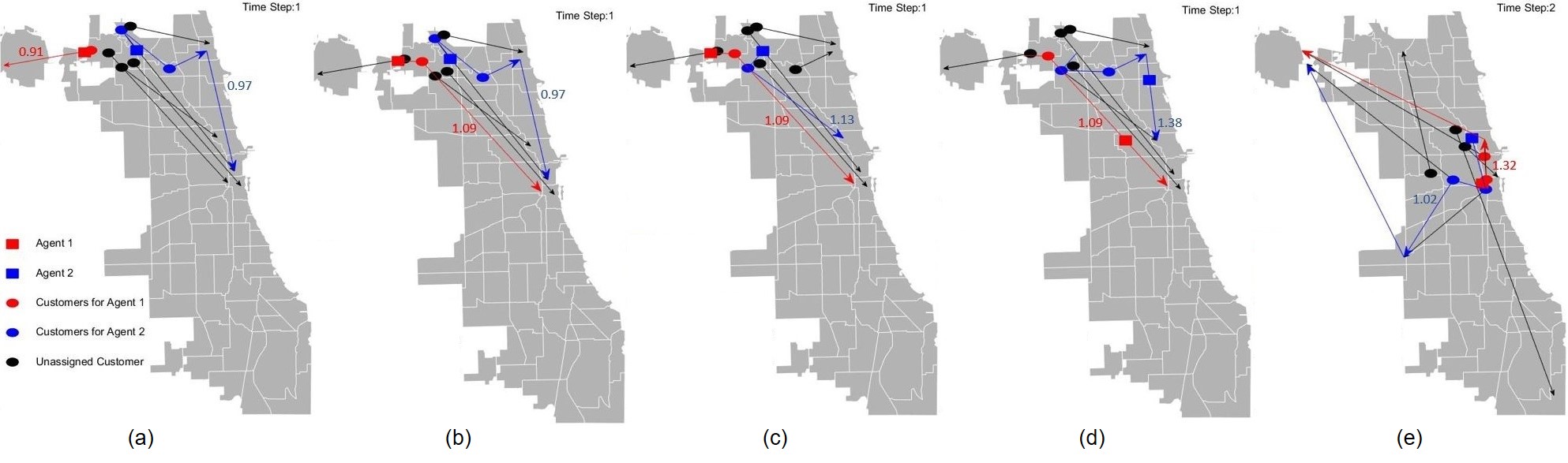}}}
\vspace{0cm} \caption{The task assignment algorithm proposed in Sec. \ref{Sec:game} is illustrated for two agents} \vspace{-0.3cm}
 \label{timeframe}\end{center}
\end{figure*}

\begin{algorithm} 
\caption{Simulation Procedure}
\begin{algorithmic}[1] 
\label{alg:sim_alg}
\State Initialize agents with ${Q}_t^i = {Q}_\mathrm{MDP}, {f}_t^i, {g}_t^i, {\omega}_t^i, \alpha_t^i$
\While{true}
\Statex \texttt{Information Seen by Agents}
\For{ $\forall$ agent $\in$ agents}
	\State Record tasks, $T(c_p^j, c_d^l)$, within $r_\mathrm{tasks}$ of agent
    \State Record agents within $R_{comm}$ of agent
	\EndFor
\Statex \texttt{Task Assignment}
\While{ $\neg  \forall  u_t^i$ converged}
 	\State Randomly select agent, $i$, with current action $u_t^i$
    \State Compute $J$ with \eqref{futilty}
    \State Compute $J_p$ with \eqref{futilty}
    \State Change action with probability: $P_i^{(u'^i_t, u^{-i}_{t - 1})}$, \eqref{log} 
    \EndWhile
\Statex \texttt{Learning Update}
\For{$\forall$ agent $\in$ agents}
	\State Update $r_t^i$ and $\boldsymbol{r_t^i}$ with \eqref{Dist-Sarsa}
    \State Update ${\omega}_t^i, \hat{{f}_t^i}, \hat{{g}_t^i}, \alpha_t^i$ with \eqref{Dist-alpha}. 
    \State Update $\boldsymbol{Q}_{t+1}^i$ with \eqref{Dist-Sarsa}. 
	\EndFor	
\State {Agents execute assigned tasks}
\EndWhile

\end{algorithmic} 
\end{algorithm}  

\indent We illustrate the algorithm in practice in Fig.~\ref{timeframe}. In (a)-(c), agents are iteratively running the algorithm to select their customers, where the marginal utility function $J^i$ for each selected customer is shown. In (d), agents are picking up and dropping off their selected customers with maximum utility function. In (e), agents are at their destinations after accomplishing their tasks and observing local new customers. Then, they execute the game theory task assignment again. 

\begin{figure}[t]
\begin{center} 
\vspace{-00cm}
{\scalebox{.95}
{\includegraphics[trim={2cm 6.25cm 2cm 7cm}, clip = true, width=9cm]{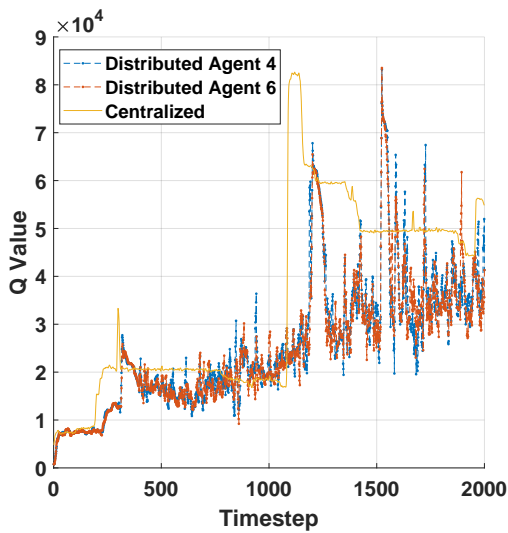}}}
\vspace{0cm}
\caption{The Q-values for a state-action pair tracking for centralized and distributed algorithms.} 
\label{Q-dis}
\vspace{-00cm}
\end{center}
\end{figure}

\indent The first validation of proposed distributed SARSA RL method is to compare distributed and centralized Q-value tracking. Figure~\ref{Q-dis} shows the Q-value, $Q(8,a^8_8)$. As expected, we see the distributed estimate approach the centralized algorithm's Q-value. It is also easy to see that agents estimates are converging and consensus is achieved. The non-zero error at large time-value is captured in our model by the error bound from Theorem \ref{proposition-bounded-optimality}, $\Delta$. This simulation is run with ten agents over the equivalent of two weeks. 

\indent The second validation of the proposed distributed SARSA RL method is to compare the total revenue generated by distributed and centralized policies, $D_\mathrm{Distributed}$ and $D_\mathrm{Centralized}$, respectively. The ratio of the generated revenue is plotted against number of agents in Fig.~\ref{TotalRevenue_fig} with varying radius of communication between agents. 

\begin{figure}[t]
\begin{center} 
\hspace{-0.5cm}
{\scalebox{1}
{\includegraphics[trim={2cm 7cm 2cm 7cm}, clip = true, width=9cm]{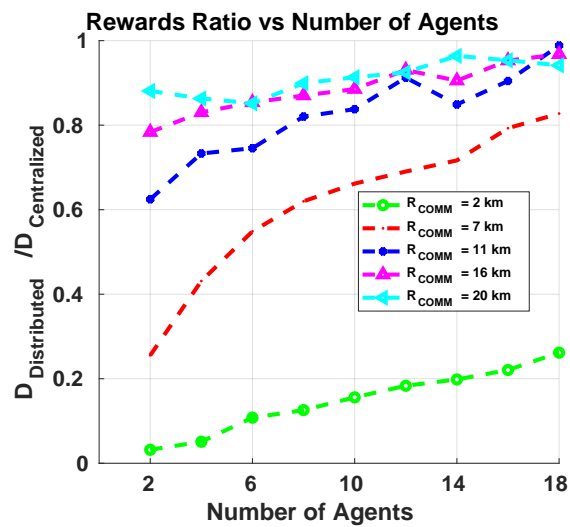} }}
\caption{Total revenue vs number of agents with varying communication radius. Each data point is a simulation.} 
\label{TotalRevenue_fig}
\end{center}
\end{figure}

\indent Figure~\ref{TotalRevenue_fig} reveals a few important effects. First, as the radius of communication increases, the revenue ratio approaches one. This corresponds to the effect that if every agent can communicate with every other agent (a complete graph), the distributed solution for every agent will converge to the centralized solution and we will recover a revenue ratio of one. The second effect we observe is that as the number of agents increases, the ratio approaches one.  All held  equal, number of agents increase would increase the estimation error from consensus. However in this case, increasing the number of agents decreases the estimation error, since the  connectivity  of  the graph is improving. So there are two competing  effects determining the performance relative to number of agents. The trends discussed are expected and validate our algorithm in  a numerical simulation. 

\begin{figure}[t]
\begin{center} 
{\scalebox{0.9}
{\includegraphics[trim={2cm 6.5cm 2cm 7cm}, clip = true, width=10cm]{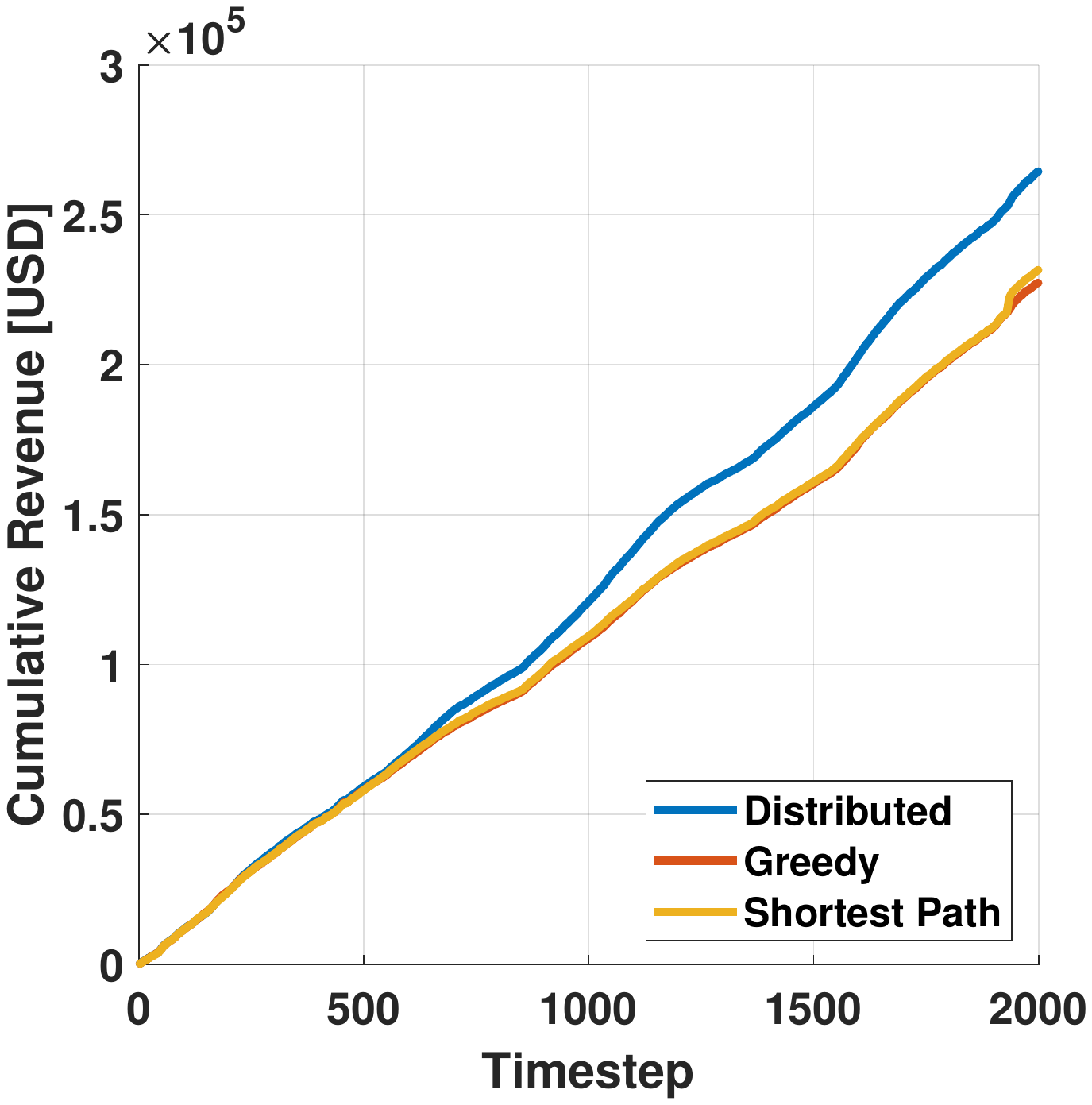} }}
\caption{Cumulative revenue at each time step for proposed algorithm, greedy, and shortest path. }
\vspace{-0.3cm}
\label{revenue_improvement_figure}
\end{center}
\end{figure}

\indent Figure~\ref{revenue_improvement_figure} demonstrates the economic utility of our proposed algorithm. This simulation is for 20 agents with a communication radius of 5.5 km. This is the cumulative reward of each algorithm. We define a 'greedy' algorithm where agents value each trip from the immediate reward. We also define a 'shortest path' algorithm where agents value each trip from the how close the request is to the agent's current location (current algorithm for most of transportation companies). Both these algorithms have no forecasting ability. At early time-steps, the distributed SARSA RL algorithm performs similarly to the 'greedy' and 'shortest path' algorithms, but outperforms these algorithms overtime because the agents are collectively updating information on a changing environment by estimating future values using the Q-value formulation. Over the equivalent of two weeks, the average return of each trip for the greedy algorithm is 10.52 USD and the average return for our distributed SARSA RL algorithm is 12.03 USD. 

Our method could be further improved by parallel advances in intelligent transportation. For example, \cite{coulomb_law} presents a physics-inspired method to increase net revenue by recommending routes for taxis with no requests by modeling the passengers and taxis as positive or negative charges.

\section{Conclusion}
In this paper, real-time distributed learning-based algorithms with guaranteed convergence properties were presented to solve the optimal transportation planning problem of autonomous vehicles for flexible-route service using autonomous taxis and ride-sharing vehicles. The proposed optimal traffic planning approach employs a distributed SARSA algorithm to allow each vehicle to use only local information and local interactions to update the Q-values. Those Q-values reflect the estimate of the company's profit for selecting different customers over a period of time. An MDP model was used to find the initial values for SARSA reinforcement learning to provide faster convergence and to guarantee a near-optimal policy before convergence.
	Furthermore, to capture the environment changes, such as the number of the customers in each area, traffic, and fares, an optimal adaptive learning rate is introduced for distributed SARSA updates. In a single agent scenario, the agent simply selects the customer with the largest value. However, in a multi-vehicle scenario, agents are required to reach an agreement on selected customers. Hence, a game-theory-based task assignment algorithm was presented, where each agent used the high-level recommendations, provided by distributed SARSA, to select the optimal customer from the set of local available requests in a distributed manner. It is proven that the introduced game is a potential game, and that agents converge to stochastically stable actions, a Nash equilibrium, if they all adhere to binary log-linear learning.
	Furthermore, a utility function was proposed to consider ride-pooling for customers, where it reduced the cost for customers and the number of required vehicles for the transportation company.
	Finally, the customers data provided by the city of Chicago was used to validate the proposed algorithms. It is shown that the proposed algorithm is highly scalable due to its distributed nature. The results of the numerical simulation validate the proposed algorithm by demonstrating that the complete graph distributed solution converges to the centralized solution. The economic utility of the algorithm is demonstrated to outperform the existing 'greedy' methods. 
    
\section*{Acknowledgment}
\textcolor{blue}{}
The authors thank the feedback from colleagues in the Data-driven Intelligent Transportation workshop (DIT 2018, held in conjunction with IEEE ICDM). The authors also thank Suzanne Olivier for her contributions on Fig.~\ref{timeframe} and the initial simulation design.}

\bibliographystyle{IEEEtran}
\bibliography{IEEEabrv,main}

\begin{appendices}
\section{} \label{Appendix-errorbound}
Using the MDP model defined in Sec.~\ref{SubSec:MDP-taxi} for a non-stationary environment creates an accumulated error at each time $t$.
To determine the upper bound  of this accumulated error, we need to introduce some definitions and assumptions.
	Here, we call the tuple $<S,A,\mathcal{P},\mathcal{R}>$, defined in Sec. \ref{SubSec:MDP-taxi}, as the stationary MDP. 
	The true model of the time-varying system at time $t$ is denoted by the tuple $<S,A,\{\mathcal{P}\}_{t=1}^\infty,\{\mathcal{R}\}_{t=1}^\infty>, \forall t$. Also, the optimal $Q$-value for stationary model and true time-varying model are denoted by $Q^*_{\text{MDP}}$ and $Q_t^*$, respectively.
    
\begin{definition} \label{k-approximate}
A sequence of random variable $X_t$, $\kappa-$approximates random variable $Y$ with $\kappa\geq 0$, if we have
\begin{align*}
\mathbb{P}\bigg(\lim_{t\to \infty} \sup \norm{X_t-Y}\leq \kappa \bigg)=1.
\end{align*}

\end{definition}

\begin{assumption}
There exist two constants $\epsilon$ and $\delta$ such that
\begin{align*}
\lim_{t\to \infty} \sup \norm{\mathcal{P}-\mathcal{P}_t} \leq \epsilon,\\
\lim_{t\to \infty} \sup \norm{\mathcal{R}-\mathcal{R}_t} \leq \delta .
\end{align*}
\end{assumption}

Now, based on the theorem introduced in \cite{changing}, we are able to obtain the upper bound error of using the stationary model while the system is changing at each time $t$.

\begin{theorem}\label{upperbound}
$Q_t$ sequence is  $\kappa$-approximate of $Q^*_{\text{MDP}}$, where 
\begin{align*}
\kappa=\frac{4 d(\epsilon,\delta)}{1-\gamma} \\
d(\epsilon,\delta)=\frac{\epsilon \gamma \norm{\mathcal{R}}_\infty}{(1-\gamma)^2}+\frac{\delta}{1-\gamma}
\end{align*}
\end{theorem}
where $\epsilon$ and $\delta$ are the bounds defined in Definition \ref{k-approximate}, $\gamma$ is the discount factor defined in \eqref{maximization}, and $\kappa$ is the $\kappa$-approximate defined in Definition \ref{k-approximate}. Theorem \ref{upperbound} provides an upper bound on the error that might be obtained if the stationary model Q-values $Q^*_{\text{MDP}}$ is used.

To justify the importance of tracking the environment changes in our problem framework, the Chicago city data is used. The stationary model, $Q^*_{\text{MDP}}$, is obtained using May 2017 data in Sec. \ref{Sec:sim}. After adding the trip information of June $1^{\text{st}} 2017$, the new model parameters $\mathcal{P}_t$ and $\mathcal{R}_t$ are calculated, where we have $\epsilon=0.2, \delta=25.4, \gamma=0.8,$ and $\|R_\infty\|=128.6$. Then by using Theorem \ref{upperbound}, we have $d(\epsilon,\delta)=641.4$ and $\kappa=12828$, where it is easy to see that the upper bound is significantly large. Thus, using only the stationary model for all time $t$ is not enough, and we need to track the changes of the environment in our problem.

\end{appendices}

\end{document}